\documentclass[a4paper]{article}

\usepackage{amsmath,amsthm,amssymb}
\usepackage{graphicx,subcaption, tikz}
\usepackage{comment,xspace,paralist}							                              
\usepackage[shortlabels]{enumitem}	
\usepackage{hyperref}       
\hypersetup{
    colorlinks=true,
    citecolor = blue,
    linkcolor=black,
    filecolor=magenta,
    urlcolor=cyan,
    bookmarks=true,
}
\usepackage{times}
\usepackage{todonotes}
\usepackage[ruled,linesnumbered]{algorithm2e}
\newtheorem{theorem}{Theorem}
\newtheorem{lemma}{Lemma}
\newtheorem{claim}{Claim}

\title{{\bf  $k$-Critical Graphs in $P_5$-Free Graphs}}

\author{
Kathie Cameron\thanks{Department of Mathematics, Wilfrid Laurier University, Waterloo, ON, Canada N2L 3C5. Email: \texttt{kcameron@wlu.ca}.
Research supported by the Natural Sciences and Engineering Research Council of Canada (NSERC) grant RGPIN-2016-06517.}
\and
Jan Goedgebeur \thanks{Department of Applied Mathematics, Computer Science and Statistics, Ghent University, 9000 Ghent, Belgium.} \thanks{Computer Science Department, University of Mons, 7000 Mons, Belgium.} \thanks{Supported by a Postdoctoral Fellowship of the Research Foundation Flanders (FWO).} 
\and
Shenwei Huang\thanks{College of Computer Science, Nankai University, Tianjin 300350, China. Partially supported by the National Natural Science Foundation of China (11801284).}
\and
Yongtang Shi\thanks{Center for Combinatorics and LPMC, Nankai University, Tianjin 300071, China. Partially supported by the National Natural Science Foundation of China (Nos. 11771221, 11922112) and the Fundamental Research Funds for the Central Universities, Nankai University.}
}

\date{May 4, 2020}

\begin{document}

\maketitle

\begin{abstract} 
Given two graphs $H_1$ and $H_2$, a graph $G$ is $(H_1,H_2)$-free if it contains no induced subgraph isomorphic to $H_1$ or $H_2$. 
Let $P_t$ be the path on $t$ vertices.
A graph $G$ is $k$-vertex-critical if $G$ has chromatic number $k$ but every proper induced subgraph of $G$ has chromatic number less than $k$.
The study of $k$-vertex-critical graphs for graph classes is an important topic in algorithmic graph theory because
if the number of such graphs that are in a given hereditary graph class is finite, then there is a polynomial-time algorithm
to decide if a graph in the class is $(k-1)$-colorable. 

In this paper, we initiate a systematic study of the finiteness of $k$-vertex-critical graphs in subclasses of $P_5$-free graphs.
Our main result is a complete classification of the finiteness of $k$-vertex-critical graphs in the class of $(P_5,H)$-free graphs 
for all graphs $H$ on 4 vertices. To obtain the complete dichotomy, we prove the finiteness for four new graphs $H$ using
various techniques -- such as Ramsey-type arguments and the dual of Dilworth's Theorem -- that may be of independent interest.
\end{abstract}

\section{Introduction}

All graphs in this paper are finite and simple.
We say that a graph $G$ {\em contains} a graph $H$ if $H$ is
isomorphic to an induced subgraph of $G$.  A graph $G$ is
{\em $H$-free} if it does not contain $H$. 
For a family of graphs $\mathcal{H}$,
$G$ is {\em $\mathcal{H}$-free} if $G$ is $H$-free for every $H\in \mathcal{H}$.
When $\mathcal{H}$ consists of two graphs, we write
$(H_1,H_2)$-free instead of $\{H_1,H_2\}$-free.
As usual, $P_t$ and $C_s$ denote
the path on $t$ vertices and the cycle on $s$ vertices, respectively. The complete
graph on $n$ vertices is denoted by $K_n$.
The graph $K_3$ is also referred to as the {\em triangle}.
For two graphs $G$ and $H$, we use $G+H$ to denote the \emph{disjoint union} of $G$ and $H$.
For a positive integer $r$, we use $rG$ to denote the disjoint union of $r$ copies of $G$.
The \emph{complement} of $G$ is denoted by $\overline{G}$.
A {\em clique} (resp.\ {\em independent set}) in a graph is a set of pairwise adjacent (resp.\ nonadjacent) vertices.
If a graph $G$ can be partitioned into $k$ independent sets $S_1,\ldots,S_k$ such that there is an edge between every vertex in $S_i$
and every vertex in $S_j$ for all $1\le i<j\le k$, $G$ is called a {\em complete $k$-partite graph}; each $S_i$ is called a {\em part}
of $G$. If we do not specify the number of parts in $G$, we simply say that $G$ is a {\em complete multipartite graph}.
We denote by $K_{n_1,\ldots,n_k}$ the complete $k$-partite graph such that the $i$th part $S_i$ has size $n_i$, for each $1\le i\le k$.
A \emph{$q$-coloring} of a graph $G$ is a function $\phi:V(G)\longrightarrow \{ 1, \ldots ,q\}$ such that
$\phi(u)\neq \phi(v)$ whenever $u$ and $v$ are adjacent in $G$.
Equivalently, a $q$-coloring of $G$ is a partition of $V(G)$ into $q$ independent sets.
A graph is {\em $q$-colorable} if it admits a $q$-coloring.
The \emph{chromatic number} of a graph $G$, denoted by
$\chi (G)$, is the minimum number $q$ for which $G$ is $q$-colorable.
The \emph{clique number} of $G$, denoted by $\omega(G)$, is the size of a largest clique in $G$.

A graph $G$ is {\em $k$-chromatic} if $\chi(G)=k$. We say that $G$ is {\em $k$-critical} if it is
$k$-chromatic and $\chi(G-e)<\chi(G)$ for any edge $e\in E(G)$. For instance, $K_2$ is the only 2-critical graph
and odd cycles are the only 3-critical graphs. A graph is {\em critical} if it is $k$-critical for some integer $k\ge 1$.
Critical graphs were first defined and studied by Dirac~\cite{Di51,Di52,Di52i} in the early 1950s,
and then by Gallai and Ore~\cite{Ga63, Ga63i,Or67} among many others, and more recently by
Kostochka and Yancey~\cite{KY14}. 

A weaker notion of criticality is the so-called vertex-criticality.
A graph $G$ is {\em $k$-vertex-critical} if $\chi(G)=k$ and $\chi(G-v)<k$ for any $v\in V(G)$.
For a set $\mathcal{H}$ of graphs and a graph $G$, we say that $G$ is {\em $k$-vertex-critical $\mathcal{H}$-free}
if it is $k$-vertex-critical and $\mathcal{H}$-free.
We are mainly interested in the following question.

\noindent {\bf The meta question.} Given a set $\mathcal{H}$ of graphs and an integer $k\ge 1$, 
are there only finitely many $k$-vertex-critical  $\mathcal{H}$-free graphs?

This question is important in the study of algorithmic graph theory because of the following theorem.
\begin{theorem}[Folklore]\label{thm:finiteness}
Given a set $\mathcal{H}$ of graphs and an integer $k\ge 1$,
if the set of all $k$-vertex-critical $\mathcal{H}$-free graphs is finite, then there is a polynomial-time algorithm to determine
whether an $\mathcal{H}$-free graph is $(k-1)$-colorable.  
\end{theorem}

In this paper, we study $k$-vertex-critical graphs in the class of $P_5$-free graphs. Our research is mainly motivated by the following
two results.

\begin{theorem}[\cite{HMRSV15}]\label{thm:P5infinite}
For any fixed $k\ge 5$, there are infinitely many $k$-vertex-critical $P_5$-free graphs.
\end{theorem}

\begin{theorem}[\cite{BHS09, MM12}]\label{thm:P5finite}
There are exactly $12$ 4-vertex-critical $P_5$-free graphs.
\end{theorem}

In light of \autoref{thm:P5infinite} and \autoref{thm:P5finite}, it is natural to ask which subclasses of $P_5$-free graphs 
have finitely many $k$-vertex-critical graphs for $k\ge 5$. 
For example, it was known that there are exactly 13 5-vertex-critical $(P_5,C_5)$-free graphs~\cite{HMRSV15},
and that there are finitely many 5-vertex-critical ($P_5$,banner)-free graphs~\cite{CHLS19, HLS19}, and
finitely many $k$-vertex-critical $(P_5,\overline{P_5})$-free graphs for every fixed $k$~\cite{DHHMMP17}.
Hell and Huang proved that there are finitely many $k$-vertex-critical $(P_6,C_4)$-free graphs~\cite{HH17}.
This was later generalized to $(P_t,K_{r,s})$-free graphs in the context of $H$-coloring~\cite{KP19}. 
Apart from these, there seem to be very few results on the finiteness of $k$-vertex-critical graphs for $k\ge 5$.
The reason for this, we think, is largely because of the lack of a good characterization of $k$-vertex-critical graphs. 
In this paper, we introduce new techniques into the problem and prove some new results beyond 5-vertex-criticality.

\subsection{Our Contributions} 

We initiate a systematic study on the subclasses of $P_5$-free graphs.
In particular,  we focus on $(P_5,H)$-free graphs when $H$ has small number of vertices.
If $H$ has at most three vertices,  the answer is either trivial or can be easily deduced from known results.
So we study the problem for graphs $H$ when $H$ has four vertices.
There are 11 graphs on four vertices up to isomorphism:

\begin{itemize}
\item $K_4$ and $\overline{K_4}=4P_1$;
\item $P_2+2P_1$ and $\overline{P_2+2P_1}$;
\item $C_4$ and $\overline{C_4}=2P_2$;
\item $P_1+P_3$ and $\overline{P_1+P_3}$; 
\item $K_{1,3}$ and $\overline{K_{1,3}}=P_1+K_3$;
\item $P_4=\overline{P_4}$.
\end{itemize}
The graphs $\overline{P_2+2P_1}$, $\overline{P_1+P_3}$ and $K_{1,3}$ are usually 
called  {\em diamond}, {\em paw} and {\em claw}, respectively.

One can easily answer our meta question for
some graphs $H$ using known results, e.g., Ramsey's Theorem for $4P_1$-free graphs: any $k$-vertex-critical
$(P_5,4P_1)$-free graph is either $K_k$ or has at most $R(k,4)-1$ vertices, where $R(s,t)$ is the Ramsey number, namely
the minimum positive integer $n$ such that every graph of order $n$ contains either a clique of size $s$ or an independent set of size $t$. 
However, the answer for certain graphs $H$
cannot be directly deduced from known results. 
In this paper, we prove that there are only finitely many $k$-vertex-critical $(P_5,H)$-free graphs
for every fixed $k\ge 1$ when $H$ is $K_4$, or $\overline{P_2+2P_1}$, or $P_2+2P_1$, or $P_1+P_3$. 
(Note that these results do not follow from the finiteness of $k$-vertex-critical $(P_5,\overline{P_5})$-free graphs proved in~\cite{DHHMMP17}.)
By combining our new results with known results,
we obtain a complete classification of the finiteness of $k$-vertex-critical $(P_5,H)$-free graphs when $H$ has 4 vertices.

\begin{theorem}\label{thm:main}
Let $H$ be a graph of order 4 and $k\ge 5$ be a fixed integer.  
Then there are infinitely many $k$-vertex-critical $(P_5,H)$-free graphs if and only if
$H$ is $2P_2$ or $P_1+K_3$.
\end{theorem}

To obtain the complete classification, we employ various techniques, some of which have not been used before to the best our knowledge.
For $H=K_4$, we use a hybrid approach combining the power of a computer algorithm and mathematical analysis. 
For $P_1+P_3$ and $P_2+2P_1$, we used the idea of fixed sets (that was first used in~\cite{HKLSS10} to give a polynomial-time
algorithm for $k$-coloring $P_5$-free graphs for every fixed $k$) combined with Ramsey-type arguments and the dual of Dilworth's Theorem.
We hope that these techniques could be helpful for attacking other related problems.

The remainder of the paper is organized as follows. We present some preliminaries
in \autoref{sec:pre} and prove our new results in \autoref{sec:new}. 
Finally, we give the proof of \autoref{thm:main} in \autoref{sec:classification}.

\section{Preliminaries}\label{sec:pre}

For general graph theory notation we follow~\cite{BM08}.
Let $G=(V,E)$ be a graph.  If $uv\in E$, we say that $u$ and $v$ are {\em neighbors} or {\em adjacent}; otherwise
$u$ and $v$ are {\em nonneighbor} or {\em nonadjacent}.
The \emph{neighborhood} of a vertex $v$, denoted by $N_G(v)$, is the set of neighbors of $v$.
For a set $X\subseteq V(G)$, let $N_G(X)=\bigcup_{v\in X}N_G(v)\setminus X$.
We shall omit the subscript whenever the context is clear.
For $X,Y\subseteq V$, we say that $X$ is \emph{complete} (resp.\ \emph{anticomplete}) to $Y$
if every vertex in $X$ is adjacent (resp.\ nonadjacent) to every vertex in $Y$.
If $X=\{x\}$, we write ``$x$ is complete (resp.\ anticomplete) to $Y$'' instead of ``$\{x\}$ is is complete (resp.\ anticomplete) to $Y$''.
If a vertex $v$ is neither complete nor anticomplete to a set $S$, we say that $v$ is {\em mixed} on $S$.
We say that $H$ is a {\em homogeneous set} if no vertex in $V-H$ is mixed on $H$.
A vertex is \emph{universal} in $G$ if it is adjacent to all other vertices.
A vertex subset $K\subseteq  V$ is a \emph{clique cutset} if $G-K$ has more components than $G$ and $K$
induces a clique.  For $S\subseteq V$, the subgraph \emph{induced} by $S$, is denoted by $G[S]$.
A {\em $k$-hole} in a graph is an induced cycle $H$ of length $k\ge 4$. If $k$ is odd, we say that $H$ is an {\em odd hole}.
An {\em $k$-antihole} in $G$ is a $k$-hole in $\overline{G}$. Odd antiholes are defined analogously.
The graph obtained from $C_k$ by adding a universal vertex, denoted by $W_k$, is called {\em $k$-wheel}.

\medskip
\noindent {\bf List coloring.} Let $[k]$ denote the set $\{1,2,\ldots,k\}$.
A {\em $k$-list assignment} of a graph $G$ is a function $L:V(G)\rightarrow 2^{[k]}$.
The set $L(v)$, for a vertex $v$ in $G$, is called the {\em list} of $v$.
In the {\em list $k$-coloring} problem, we are given a graph $G$ with a $k$-list assignment $L$ and asked 
whether $G$ has an {\em $L$-coloring}, i.e., a $k$-coloring of $G$ such that every
vertex is assigned a color from its list. We say that $G$ is {\em $L$-colorable} if $G$ has an $L$-coloring.
If the list of every vertex is $[k]$, then the list $k$-coloring problem is precisely
the $k$-coloring problem. 

A common technique in the study of graph coloring is called {\em propagation}. If a vertex $v$ has its color forced to be $i\in [k]$,
then no neighbor of $v$ can be colored with color $i$. This motivates the following definition. 

Let $(G,L)$ be an instance of the list $k$-coloring problem. The color of a vertex $v$ is said to be {\em forced} if $|L(v)|=1$.
A {\em propagation from a vertex $v$} with $L(v)=\{i\}$ is the procedure of removing $i$ from the list of every neighbor of $v$.
If we denote by the resulting $k$-list assignment by $L'$, then $G$ is $L$-colorable if and only if $G-v$ is $L'$-colorable.  
A propagation from $v$ could make the color of other vertices forced; if we continue to propagate from those vertices until
no propagation is possible, we call the procedure ``{\em exhaustive propagation from $v$}''.
It is worth mentioning that the idea of propagation is featured in many recent studies on coloring $P_t$-free graphs and related problems, see 
\cite{BCMSSZ17, CGSZ20} for example.

\medskip
\noindent {\bf An example of propagation.} Let $G$ be a 4-vertex path $w,x,y,z$ with $L(w)=\{1\}$, $L(x)=\{1,2\}$, $L(y)=\{2,3\}$, and $L(y)=\{1,2\}$.
Then propagation from $w$ results in the new list assignment $L'$ where $L'(x)=\{2\}$ and $L'(v)=L(v)$ for $v\neq x$.
On the other hand, exhaustive propagation from $w$ results in the new list assignment $L''$ where 
$L''(w)=\{1\}$, $L''(x)=\{2\}$, $L''(y)=\{3\}$, $L''(z)=\{1,2\}$.

\medskip
We proceed with a few useful results that will be needed later. The first one is a folklore property of $k$-vertex-critical graphs.

\begin{lemma}[Folklore]\label{lem:clique cutsets}
Any $k$-vertex-critical graph cannot contain clique cutsets.
\end{lemma}

Another folklore property of vertex-critical graphs is that such graph cannot contain two nonadjacent vertices $u,v$ such that
$N(v)\subseteq N(u)$. We generalize this property to anticomplete subsets.

\begin{lemma}\label{lem:dominated subsets}
Let $G$ be a $k$-vertex-critical graph. Then $G$ has no two nonempty disjoint subsets $X$ and $Y$ of $V(G)$ that satisfy all the following conditions.

\begin{itemize}
\item $X$ and $Y$ are anticomplete to each other.
\item $\chi(G[X])\le \chi(G[Y])$.
\item $Y$ is complete to $N(X)$.
\end{itemize}
\end{lemma}

\begin{proof}
Suppose that $G$ has a pair of nonempty subsets $X$ and $Y$ that satisfy all three conditions.
Since $G$ is $k$-vertex-critical, $G-X$ has a $(k-1)$-coloring $\phi$.  Let $t=\chi(G[Y])$.
Since $Y$ is complete to $N(X)$, at least $t$ colors do not appear on any vertex in $N(X)$ under $\phi$.
So we can obtain a $(k-1)$-coloring of $G$ by coloring $G[X]$ with those $t$ colors. This contradicts
that $G$ is $k$-chromatic.
\end{proof}

A graph $G$ is {\em perfect} if $\chi(H)=\omega(H)$ for each induced subgraph $H$ of $G$.
An {\em imperfect} graph is a graph that is not perfect.
A classical theorem of Dilworth~\cite{Dil50} states that the largest size of an antichain in a partially ordered set is equal to 
the minimum number of chains that partition the set. We will use the dual of Dilworth's Theorem which says that
the largest size of a chain in a partially ordered set is equal to  the minimum number of antichains that partition the set.
This was first proved by Mirsky~\cite{Mir71} and it has an equivalent graph-theoretic interpretation via comparability graphs.
 A graph is a {\em comparability graph}
if the vertices of the graph are elements of a partially ordered set and 
two vertices are connected by an edge if and only if the corresponding elements are comparable.

\begin{theorem}[Dual Dilworth Theorem~\cite{Mir71}]\label{thm:Dual Dilworth}
Every comparability graph is perfect.
\end{theorem}

We conclude this section with the celebrated Strong Perfect Graph Theorem~\cite{CRST06}.

\begin{theorem}[Strong Perfect Graph Theorem~\cite{CRST06}]\label{thm:SPGT}
A graph is perfect if and only if it contains no odd holes or odd antiholes.
\end{theorem}

\section{New Results}\label{sec:new}

In this section, we prove four new results: there are finitely many $k$-vertex-critical $(P_5,H)$-free graphs
when $H\in \{K_4,\overline{P_2+2P_1},P_2+2P_1,P_1+P_3\}$.

\subsection{$K_4$-Free Graphs}

Let $G_1$ be the 13-vertex graph with vertex set $\{0,1,\ldots,12\}$ and the following edges:
\begin{itemize}
\item $\{3,4,5,6,7\}$ and $\{0,1,2,8,9\}$ induce two disjoint 5-holes $Q$ and $Q'$;
\item 12 is complete to $Q\cup Q'$;
\item 11 is complete to $Q$ and 10 is complete to $Q'$ with 10 and 11 being connected by an edge.
\end{itemize}

Let $G_2$ be the 14-vertex graph with vertex set $\{0,1,\ldots,13\}$ and the following edges:
\begin{itemize}
\item $\{12,13\}$ is a cutset of $G_2$ such that 12 and 13 are not adjacent and $G_2-\{12,13\}$ has exactly two components;
\item One component of $G_2-\{12,13\}$ is a 5-hole induced by $\{0,1,2,3,4\}$, and this 5-hole is complete to $\{12,13\}$;
\item The other component, induced by $\{5,6,7,8,9,10,11\}$, is the graph in \autoref{fig:ruby}, and 12 is complete to $\{5,8,9,10,11\}$
and 13 is complete to $\{6,7,9,10,11\}$.
\end{itemize}
The adjacency lists of $G_1$ and $G_2$ are given in the Appendix. It is routine to verify that $G_1$ and $G_2$ are 5-vertex-critical
$(P_5,K_4)$-free graphs. The main result in this subsection is that they are the only 5-vertex-critical $(P_5,K_4)$-free graphs. 

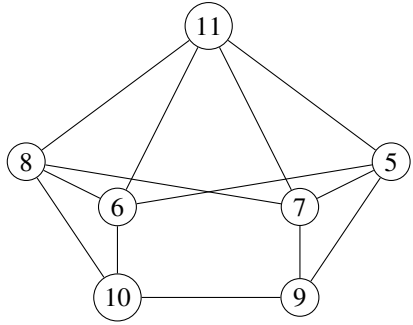
\begin{figure}[t]
\centering

\begin{tikzpicture}[scale=0.6]
\tikzstyle{vertex}=[draw, circle, fill=white!100, minimum width=4pt,inner sep=2pt]

\node[vertex] (11) at (0,6) {11};
\node[vertex] (5) at (4,3) {5};
\node[vertex] (9) at (2,0) {9};
\node[vertex] (10) at  (-2,0) {10};
\node[vertex] (8) at (-4,3) {8};

\draw (11)--(5)--(9)--(10)--(8)--(11);

\node[vertex] (7) at (2,2) {7};
\node[vertex] (6) at (-2,2) {6};
\draw (8)--(7)--(11);
\draw (5)--(7)--(9);
\draw (8)--(6)--(10);
\draw (5)--(6)--(11);
\end{tikzpicture}

\caption{One component of $G_2-\{12,13\}$.}\label{fig:ruby}

\end{figure}

\begin{theorem}\label{thm:K4}\label{thm:K4oddhole}
Let $G$ be a 5-vertex-critical $(P_5,K_4)$-free graph. Then $G$ is isomorphic to either $G_1$ or $G_2$.
\end{theorem}

We will prove \autoref{thm:K4} in a series of intermediate steps.
We will need the following result.
\begin{theorem}[\cite{CRST10}]
Any $K_4$-free graph with no odd holes is 4-colorable.
\end{theorem}

The next two lemmas are based on a computer generation approach to exhaustively generate all $k$-vertex-critical graphs in a given class of 
$\mathcal{H}$-free graphs via a recursive algorithm. The idea of computer generation was first used in~\cite{HMRSV15}, and later developed extensively by 
Goedgebeur and Schaudt~\cite{GO18} and Chudnovsky et al.~\cite{CGSZ20JCTB}.

We say that $G'$ is a {\em 1-vertex extension} of $G$ if $G$ can be obtained from $G'$ by deleting a vertex in $G'$.
Roughly speaking, the generation algorithm starts with some small substructure which must occur in any $k$-vertex-critical graph, and then 
exhaustively searches for all 1-vertex extensions of the substructure.  The algorithm stores
those extensions that are $k$-vertex-critical and $\mathcal{H}$-free in the output list $\mathcal{F}$.
Then it recursively repeats the procedure for all $(k-1)$-colorable substructures found in the previous iterations.
The pesudocode of the generation algorithm is given in \autoref{alg:main} and \autoref{alg:extend}.

It should be noted that with a naive implementation the algorithm may not terminate.
For instance, if we extend a graph $G$ by repeatedly adding vertices that have the same neighborhood as a vertex in $G$, the program will never terminate. So one has to design certain pruning rules to make the algorithm terminate. 
For instance, if $G$ contains two nonadjacent vertices $u,v$ such that $N(u)\subseteq N(v)$, then we only need to consider all 1-vertex extensions
$G'$ such that the unique vertex in $V(G')\setminus V(G)$ is adjacent to $u$ but not adjacent to $v$ (by \autoref{lem:dominated subsets}). 
In~\cite{HMRSV15}, the authors designed two pruning rules like this so that the algorithm terminates with 13 5-vertex-critical $(P_5,C_5)$-free
graphs.  Later, the technique was extensively developed by Goedgebeur and Schaudt~\cite{GO18} who introduced many more useful pruning rules
that are essential for generating all critical graphs in certain classes of graphs, e.g., 4-vertex-critical $(P_7,C_4)$-free graphs and 
4-vertex-critical $(P_8,C_4)$-free graphs.  The word ``valid'' in \autoref{alg:extend} is used precisely to quantify those extensions that
survive a specific set of pruning rules. 

The algorithm we use in this paper is exactly the one developed in~\cite{GO18}. Hence, the valid extensions on line 8 in 
Extend($G$) are with respect to all pruning rules given in Algorithm~2 in~\cite{GO18} (since we only use those rules as a black box, we do not define them here).
\begin{theorem}[\cite{GO18}]\label{thm:termination}
If  \autoref{alg:main} terminates and returns the list $\mathcal{F}$, then $\mathcal{F}$ is exactly the set  of all k-vertex-critical 
$\mathcal{H}$-free graphs containing $S$.
\end{theorem}

\begin{algorithm}[H]

\SetAlgoLined

\KwIn{An integer $k$, a set $\mathcal{H}$ of forbidden induced subgraphs, and a graph $S$.}

\KwOut{A list $\mathcal{F}$ of all $k$-vertex-critical $\mathcal{H}$-free graphs containing $S$.}

Let $\mathcal{F}$ be an empty list.

Extend($S$).

Return $\mathcal{F}$.

\medskip
\caption{Generate($k$, $\mathcal{H}$, $S$)} \label{alg:main}
\end{algorithm}

\begin{algorithm}[H]

\SetAlgoLined

\If{{\em $G$ is $\mathcal{H}$-free and is not generated before}}
{
	\If{$\chi(G)\ge k$}
	{
		\If{{\em $G$ is $k$-vertex-critical}}	{add $G$ to $\mathcal{F}$}
	}
	\Else
	{
		\For{{\em each {\color{blue} valid} 1-vertex extension $G'$ of $G$}}
		{
			Extend($G'$)
		}
	}
}

\medskip
\caption{Extend($G$)} \label{alg:extend}
\end{algorithm}

Let $F$ be the graph obtained from a 5-hole by adding a new vertex and making it adjacent to four vertices on the hole.

\begin{lemma}\label{lem:W5F}
Let $G$ be a 5-vertex-critical $(P_5,K_4)$-free graph. If $G$ contains an induced $W_5$ or $F$, then 
 Then $G$ is isomorphic to either $G_1$ or $G_2$.
\end{lemma}

\begin{proof}
We run \autoref{alg:main} with the following inputs:
\begin{itemize}
\item $k=5$;
\item $\mathcal{H}=\{P_5,K_4\}$;
\item $S=W_5$ or $S=F$.
\end{itemize}
If $S=W_5$, then the algorithm terminates with the graphs $G_1$ and $G_2$,
and if $S=F$, then it terminates with only the graph $G_2$. The correctness of the algorithm follows from \autoref{thm:termination}.
\end{proof}

\begin{lemma}\label{lem:7-antihole-free}
Let $G$ be a 5-vertex-critical $(P_5,K_4)$-free graph. If $G$ is 7-antihole-free, then $G$ is isomorphic to $G_1$.
\end{lemma}

\begin{proof}
By \autoref{thm:K4oddhole}, $G$ must contain a 5-hole. We run \autoref{alg:main} with the following inputs:
\begin{itemize}
\item $k=5$;
\item $\mathcal{H}=\{P_5,K_4,\overline{C_7}\}$;
\item $S=C_5$
\end{itemize}
The algorithm terminates and outputs $G_1$ as the only critical graph. The correctness of the algorithm follows from \autoref{thm:termination}.
\end{proof}

\begin{lemma}\label{lem:7-antihole}
Let $G$ be a $(P_5,K_4,W_5,F)$-free graph. If $G$ contains an 7-antihole, then $G$ is 4-colorable.
\end{lemma}

\begin{proof}
Let $C=v_1,v_2,\ldots,v_7$ be a 7-antihole with $v_iv_{i+1}$ being a nonedge. 
For each $1\le i\le 7$, 
let $T_i$ be the set of vertices in $V\setminus V(C)$ that are adjacent to $v_{i-1},v_i,v_{i+1}$,
and $F_i$ be the set of vertices in $V\setminus V(C)$ that are adjacent to $V(C)\setminus \{v_{i-1},v_i,v_{i+1}\}$.

\begin{claim}\label{clm:vertex type}
$V\setminus V(C)=\cup_{1\le i\le 7}(F_i\cup T_i)$.
\end{claim}

\begin{proof}
Let $x\in V\setminus V(C)$ that has at least one neighbor in $C$. Since $G$ is $K_4$-free, $x$ has at most four neighbors on $C$.
Suppose first that $x$ has at most two neighbors on $C$. 
If $x$ is adjacent to $v_4$ and $v_5$, then $\{v_3,v_4,v_5,v_6,x\}$ induces a 5-hole and $v_1$ is adjacent to four vertices on the hole.
This contradicts that $G$ is $F$-free.
So $x$ cannot be adjacent only to $v_i$ and $v_{i+1}$ for some $i$.
Thus, we may assume by symmetry that $x$ is adjacent to $v_1$ and possibly to $v_{3}$ or $v_4$ (but not both). Then
$x,v_1,v_6,v_2,v_7$ is an induced $P_5$, a contradiction.
Now suppose that $x$ has three neighbors on $C$. 
Since $G$ is $K_4$-free, $x$ has at least two consecutive neighbors, say $v_1,v_2$ by symmetry. 
If $x$ is adjacent to $v_3$ or $v_7$, then $x\in T_1$ or $x\in T_2$. So $x$ is not adjacent to $v_3$ or $v_7$.
Since $x$ is adjacent to only one vertex in $\{v_4,v_5,v_6\}$, $G[\{v_3,v_4,v_5,v_6,v_7,x\}]$ contains an induced $P_5$, a contradiction.
Now suppose that $x$ has four neighbors on $C$. Then $x$ must have two consecutive neighbors, say $v_1,v_2$ by symmetry. 
If $x$ does not have three consecutive neighbors, then $x$ is not adjacent to $v_3$ or $v_7$. Then $\{v_7,v_1,v_2,v_3,x\}$ induces a $C_5$.
Since $G$ is $W_5$-free, $x$ is not adjacent to $v_5$, and so is adjacent to $v_4$ and $v_6$. But then $\{v_1,v_4,v_6,x\}$ induces a $K_4$.
Thus, $x$ is adjacent to $v_3$ or $v_7$, say $v_3$ by symmetry. If $x$ is adjacent to $v_7$ or $v_4$, then $x\in F_5$ or $x\in F_6$.
Otherwise $x$ is adjacent to exactly one of $v_5$ or $v_6$. But then $\{v_4,v_5,v_6,v_7,x\}$ induces a $P_5$. 

Now let $z\in V\setminus V(C)$ that has no neighbor in $C$. Since $G$ is connected and $P_5$-free, $z$ has a neighbor in $T$ or $F$.
If $z$ is adjacent to $t_1\in T_1$, then $z,t_1,v_2,v_5,v_3$ is an induced $P_5$. If $t_1$ is adjacent to $f_1\in F_1$,
then $z,f_1,v_3,v_7,v_2$ is an induced $P_5$. So there is no such vertex.  This proves the claim. 
\end{proof}

Note that since $G$ is $K_4$-free,  $F_i$ and $T_i$ are independent sets for each $1\le i\le 7$.
We now investigate the adjacency among $T_i$ and $F_i$ for $1\le i\le 7$.
\begin{claim}\label{clm:twin vs twin}
For each $i$, $T_i$ is anticomplete to $T_{i+1}$, and is complete to $T_{i+3}$.
\end{claim}

\begin{proof}
By symmetry, it suffices to prove the claim for $i=1$. Let $t_1\in T_1$. If $t_1$ is adjacent to $t_2\in T_2$,
then $t_2,t_1,v_7,v_4,v_6$ is an induced $P_5$. If $t_1$ is not adjacent to $t_3\in T_3$, then $\{v_2,t_3,v_3,t_1,v_1,\}$
If $t_1$ is not adjacent to $t_4\in T_4$, then $t_1,v_2,v_6,v_3,t_4$ is an induced $P_5$.
\end{proof}

\begin{claim}\label{clm:twin vs clone}
For each $i$, $F_i$ is complete to $T_{i-1}\cup T_i\cup T_{i+1}$, and anticomplete to $T_{i+3}$.
\end{claim}

\begin{proof}
Let $f\in F_1$.  Note that $C'=V(C)\setminus \{v_1\}\cup \{f\}$ induces a 7-antihole, where $f$ plays the role of $v_1$.
If $t_1\in T_1$ is not adjacent to $f$, then it is adjacent to two nonconsecutive vertices on $C'$, contradicting \autoref{clm:vertex type}. 
If $t\in T_2\cup T_7$ is not adjacent to $f$, then $t$ is adjacent to exactly two consecutive vertices on $C'$ , 
contradicting \autoref{clm:vertex type}.
If $t\in T_4$ is adjacent to $f$, the neighbors of $t$ on $C'$ are not consecutive, contradicting \autoref{clm:vertex type}. 
This proves the claim.
\end{proof}

\begin{claim}\label{clm:clone vs clone}
For each $i$, $F_i$ is anticomplete to $F_{i+1}$, and complete to $F_{i+3}$.
\end{claim}

\begin{proof}
It suffices to prove for $i=1$. Let $f\in F_1$. If $f$ is adjacent to $f'\in F_2$, then $\{f,f',v_4,v_6\}$ induces a $K_4$.
If $f$ is not adjacent to $f'\in F_4$, then the neighbors of $f'$ on $C'=V(C)\setminus \{v_1\}\cup \{f\}$ are not consecutive, 
contradicting \autoref{clm:vertex type}. 
This proves the claim.
\end{proof}

\begin{claim}\label{clm:dominated vertex}
For each $t\in T_i$, $N(t)\subseteq N(v_{i-3})\cup N(v_{i+3})$.
\end{claim}

\begin{proof}
We prove for $i=1$. Let $x$ be a common nonneighbor of $v_4$ and $v_5$.
If $x=v_4$ or $v_5$, then $x\notin N(t)$ by the definition of $t$.
If $x\in T$, then $x\in T_7\cup T_1\cup T_2$, and $x\notin N(t)$ by \autoref{clm:twin vs twin}.
If $x\in F$, then $x\in F_4\cup F_5$, and $x\notin N(t)$ by \autoref{clm:twin vs clone}.
This proves the claim.
\end{proof}

Let $L$ be the 4-list assignment of $G$ such that 

\begin{itemize}
\item $L(v_1)=\{1\}$, $L(v_2)=L(v_3)=\{2\}$, $L(v_4)=L(v_5)=\{3\}$, and $L(v_6)=L(v_7)=\{4\}$,
\item and $L(v)=[k]$ for every $v\in V\setminus V(C)$.
\end{itemize}

\begin{claim}\label{clm:eliminate three T}
$G$ is $L$-colorable if and only if $G-(T_{6}\cup T_1\cup T_{3})$ is $L$-colorable.
\end{claim}

\begin{proof}
Suppose that $G-(T_{6}\cup T_1\cup T_{3})$ has an $L$-coloring $\phi$.
We color every vertex in $T_1$ with color 3, color every vertex in $T_3$ with color 4,  and color eevery vertex in $T_6$ with color 2.
This extended coloring is an $L$-coloring of $G$ by \autoref{clm:dominated vertex}.
\end{proof}

We now prove that $G$ is $L$-colorable, which implies that $G$ is 4-colorable. 
By \autoref{clm:eliminate three T}, it suffices to show that $G-(T_6\cup T_1\cup T_3)$ is $L$-colorable.
We shall do this in a number of steps.

\medskip
\noindent {\bf The first step: propagate from $C$.} We propagate from $v_1,\ldots,v_7$.

\begin{itemize}
\item The list of every vertex in $F_1,F_3,F_4,F_5,F_6$ is $\{1\},\{2\},\{3\},\{3\},\{4\}$ respectively in this order.
Every vertex in $F_2$ has list $\{1,2\}$ and every vertex in $F_7$ has list $\{1,4\}$.
\item Every vertex in $T_2$ has list $\{3,4\}$, every vertex in $T_4$ has list $\{1,4\}$, every vertex in $T_5$ has list $\{1,2\}$,
and every vertex in $T_7$ has list $\{2,3\}$.
\end{itemize}

Let $G'$ be the subgraph of $G$ with list assignment $L'$ described in \autoref{fig:list coloring instance}.
Note that $G'$ is not an induced subgraph of $G$.
It follows from Claims \ref{clm:twin vs twin}-\ref{clm:clone vs clone} that $G$ is $L$-colorable 
if and only if  $G'$ is $L'$-colorable. (Some vertex subsets such as $F_1$ and edges such as those between $T_2$ and $F_2$
are irrelevant in terms of coloring because of either disjoint lists or nonadjacency between vertices.)

\begin{figure}[t]
\centering

\begin{tikzpicture}[scale=0.8]
\tikzstyle{vertex}=[draw, circle, fill=white!100, minimum width=4pt,inner sep=2pt]

\node[vertex] (T2) at (4,2) {$T_2:3,4$};
\node[vertex] (T4) at (4,-2) {$T_4:1,4$};
\node[vertex] (T5) at (-4,-2) {$T_5:1,2$};
\node[vertex] (T7) at (-4,2) {$T_7:2,3$};
\node[vertex] (F7) at  (-2,0) {$F_7:1,4$};
\node[vertex] (F2) at  (2,0) {$F_2:1,2$};
\node[vertex] (F3) at  (-4,-5) {$F_3:2$};
\node[vertex] (F6) at  (4,-5) {$F_6:4$};
\node[vertex] (F5) at  (-6,0) {$F_5:3$};
\node[vertex] (F4) at  (6,0) {$F_4:3$};

\draw (F5)--(T7)--(T2)--(F4);
\draw (F7)--(F2);
\draw (F3)--(T5)--(T7);
\draw (F6)--(T4)--(T2);
\draw (T7)--(F2)--(T4);
\draw (T2)--(F7)--(T5);

\end{tikzpicture}

\caption{The instance $(G',L')$.  A line between any two sets means that the edges between 
the two sets are arbitrary. No line means that edges are irrelevant in terms of coloring.}\label{fig:list coloring instance}

\end{figure}
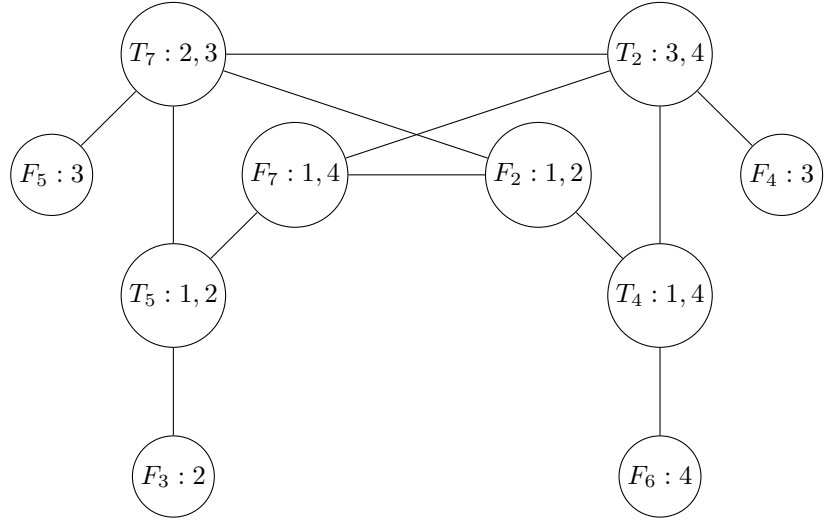

\medskip
\noindent {\bf The second step: propagate exhaustively from $F_3,F_4,F_5,F_6$.} 
We propagate the coloring from all vertices in $F_3\cup F_4\cup F_5\cup F_6$ exhaustively.

Let $T'_7=N(F_5)\cap T_7$,  $T'_5=N(T'_7\cup F_3)\cap T_5$, and $F'_7=N(T'_5)\cap F_7$. 
Since every vertex in $F_5$ has list $\{3\}$, every vertex in $T'_7$ is must be colored with 2 in any $L'$-coloring.
Similarly, every vertex in $T'_5$ and in $F'_7$ must be colored with 1 and 4, respectively. 
Symmetrically, let $T'_2=N(F_4)\cap T_2$,  $T'_4=N(T'_2\cup F_6)\cap T_4$, and $F'_2=N(T'_4)\cap F_2$. 
Then every vertex in $T'_2$ must be colored with 4, every vertex in $T'_4$ must be colored with 1, and 
every vertex in $F'_2$ must be colored with 2. 
Let $L''$ denote the resulting list assignment.
For every set $S\in \{T_2,T_4,T_5,T_7,F_2,F_7\}$, let $S''=S\setminus S'$.
Let $G''=G[T''_2\cup T''_4\cup T''_5\cup T''_7\cup F''_2\cup F''_7]$.

\begin{claim}\label{clm:TtwoF}
For each $i$, every vertex in $T_i$ is anticomplete to either $F_{i-2}$ or $F_{i+2}$ 
\end{claim}

\begin{proof}
If $t_1\in T_1$ has a neighbor $f_3\in F_3$ and a neighbor $f_6\in F_6$, then $\{t_1,f_3,f_6,v_1\}$ induces a $K_4$.
\end{proof}

\begin{claim}\label{clm:FtwoT}
For each $i$, every vertex in $F_i$ is anticomplete to either $T_{i-2}$ or $T_{i+2}$.
\end{claim}

\begin{proof}
Suppose that $f_1\in F_1$  has a neighbor $t_3\in T_3$ and a neighbor $t_6\in T_6$. 
Then $Q=v_1,v_4,t_3,t_6,v_5$ is a 5-hole with $f_1$ having four neighbors on $Q$.
This contradicts that $G$ is $F$-free.
\end{proof}

By \autoref{clm:TtwoF}, $T'_7$ and $T'_2$ are anticomplete to $F_2$ and $F_7$, respectively.
By \autoref{clm:FtwoT}, $F'_7$ and $F'_2$ are anticomplete to $T_2$ and $T_7$, respectively.
Therefore, $G'$ is $L'$-colorable if and only if $G''$ is $L''$-colorable.

\medskip
\noindent {\bf The final step: color $G''$.} 
We finish the proof by giving an $L''$-coloring of $G''$.

\begin{itemize}
\item Color every vertex in $F''_7$ with color $4$ and every vertex in $F''_2$ with color 1.
\item Assign color 4 to those vertices in $T''_4$ that are neighbors of $F''_2$, and assign color 1 to the remaining vertices in $T''_4$.
\item Assign color 3 to those vertices in $T''_2$ that are neighbors of $F''_7$ or neighbors of vertices in $T''_4$ with color 4,
and assign color 4 to the remaining vertices in $T''_2$.
\item Assign color 2 to those vertices in $T''_7$ that are neighbors of $T''_2$ with color 3, and assign color 3 to the remaining vertices in $T''_7$.
\item Assign color 1  to those vertices in $T''_5$ that are neighbors of $T''_7$ with color 2, and assign color 2 to the remaining vertices in $T''_5$.
\end{itemize}
It is routine to verify that the assignment is an $L''$-coloring of $G''$. This completes the proof.
\end{proof}

We are now ready to prove \autoref{thm:K4}.
\begin{proof}[Proof of \autoref{thm:K4}]
Let $G$ be a 5-vertex-critical $(P_5,K_4)$-free graph. If $G$ contains an induced $W_5$ or $F$, then $G$ is either $G_1$ or $G_2$
by \autoref{lem:W5F}. So we can assume that $G$ is $(W_5,F)$-free as well. By \autoref{lem:7-antihole}, $G$ must be 7-antihole-free,
and so is $G_1$ by \autoref{lem:7-antihole-free}.
\end{proof}

\subsection{$P_1+P_3$-Free Graphs}

\begin{theorem}\label{thm:co-paw}
For every fixed integer $k\ge 1$, there are finitely many $k$-vertex-critical $P_1+P_3$-free graphs.
\end{theorem}

\begin{proof}
Let $G$ be a $k$-vertex-critical $P_1+P_3$-free graph. If $G$ contains a $K_{k}$, then $G$ is isomorphic to $K_k$.
So we assume in the following that $G$ is $K_{k}$-free.
Let $K=\{v_1,\ldots,v_t\}$ be a maximal clique, where $1\le t< k$.
Since $K$ is maximal, every vertex in $V\setminus K$ is not adjacent to at least one vertex in $K$.
We partition $V\setminus K$ into the following subsets.

\begin{itemize}
\item $F_1$ is the set of nonneighbors of $v_1$.
\item  For $2\le i\le t$, $F_i$ is the set of nonneighbors of $v_i$ that are not in $F_1\cup \cdots \cup F_{i-1}$.
\end{itemize}
By the definition, $v_i$ is complete to $F_j$ if $i<j$.
Since $G$ is $P_1+P_3$-free,  each $F_i$ is $P_3$-free, and so is a disjoint union of cliques.

\begin{claim}\label{clm:neighbor}
If $F_i$ has at least two components, then every neighbor of $v_i$ is either complete or anticomplete to $F_i$.
\end{claim}

\begin{proof}
Let $v$ be a neighbor of $v_i$. Suppose that $v$ has a neighbor $f$ in $F_i$. 
Let $K$ be the component of $F_i$ containing $f$.
If $v$ is not adjacent to some vertex $f'\in F_1\setminus K$, then $\{f',f,v,v_i\}$ induces a $P_1+P_3$, a contradiction. 
So $v$ is complete to $F_i\setminus K$. Since $F_i$ has at least two components, $v$ has a neighbor in a component other than $K$.
It follows from the same argument that $v$ is complete to $K$. This completes the proof.
\end{proof}

\begin{claim}\label{clm:nonneighbor}
For every nonneighbor $v$ of $v_i$ and every component $K$ of $F_i$, $v$ is either complete or anticomplete to $K$.
\end{claim}

\begin{proof}
If $v$ is mixed on an edge $xy$ in $K$, then $\{v,v_i,x, y\}$ induces a $P_1+P_3$, a contradiction. 
\end{proof}

By \autoref{clm:neighbor} and  \autoref{clm:nonneighbor}, if $F_i$ has at least two components,
every component of $F_i$ is a homogeneous set of $G$. Moreover, since $v_i$ is complete to $F_j$ for $i<j$,
no vertex in $\{v_j\}\cup F_j$ with $j>i$ is mixed on two components of $F_i$.
We next show that each $F_i$ has bounded size.

\begin{claim}\label{clm:F1}
$|F_1|\le k$.
\end{claim}

\begin{proof}
We show that $F_1$ is connected. Suppose not. Let $K$ and $K'$ be two component of $F_1$ with $|K|\le |K'|$. 
Then $N(K)=N(K')$. By \autoref{lem:dominated subsets}, $G$ is not $k$-vertex-critical. This is a contradiction.
Therefore, $F_1$ is a clique and so has at most $k$ vertices.
\end{proof}

\begin{claim}\label{clm:Fi}
For each $1\le i\le t$, $F_i$ has bounded size.
\end{claim}

\begin{proof}
We prove this by induction on $i$. By \autoref{clm:F1}, the statement is true for $i=1$.
Now assume that $i\ge 2$ and $F_j$ has bounded size for each $1\le j<i$. If $F_i$ is connected,
then $|F_i|\le k$ and we are done. So we assume that $F_i$ has at least two components.
We will show that the number of components in $F_i$ is bounded and this will complete the proof.
For this purpose, we construct a graph $X$ as follows.

\begin{itemize}
\item $V(X)$ is the set of all components of $F_i$.
\item Two components $K$ and $K'$ of $F_i$ are connected by an edge in $X$ if and only if $N(K)\subseteq N(K')$ or $N(K')\subseteq N(K)$.
\end{itemize}

Note that $X$ is a comparability graph. Next we show that $\omega(X)\le k$. Suppose that $K_1,\ldots,K_t$ is a maximum clique in $X$
with $t>k$. We may assume that $N(K_1)\subseteq N(K_2)\subseteq \cdots \subseteq N(K_t)$. 
It follows from \autoref{lem:dominated subsets} that $|K_i|>|K_j|$ for $i<j$, i.e., $|K_1|>|K_2|>\cdots >|K_t|\ge 1$.
So $|K_1|\ge k$. This is a contradiction, since $G$ is $K_k$-free. 
This proves that $\omega(X)\le k$.
Since $X$ is perfect by \autoref{thm:Dual Dilworth},
$V(X)$ can be partitioned into at most $k$ independent sets $S_1,\ldots,S_k$. We show that each $S_p$ has bounded size.
Let $K$ and $K'$ be two components in $S_p$. Then there are vertices $x$ and $x'$ such that $x\in N(K)\setminus N(K')$ and
$x'\in N(K')\setminus N(K)$. Note that $x,x'\in T_i=\bigcup_{1\le j<i} F_j\cup \{v_j\}$. If $|S_p|>2|T_i|^2$,
by the pigeonhole principle, there are two pairs $\{K,K'\}$ and $\{L,L'\}$ of components that correspond to the same pair
$\{x,x'\}$ in $T_i$. Then $\{K,x,L,K'\}$ induces a $P_1+P_3$. This shows that each $S_p$ has size at most $2|T_i|^2$,
which is a constant by the inductive hypothesis. Therefore, $X$ has constant number of vertices, i.e., $F_i$ has constant number of
components. This completes the proof.
\end{proof}

By \autoref{clm:F1} and \autoref{clm:Fi}, each $|F_i|\le M$ for some constant $M$ (depending only on $k$).
Therefore, $G$ has bounded size. 
\end{proof}

\subsection{$P_2+2P_1$-Free Graphs}

\begin{theorem}\label{thm:co-diamond}
For every fixed integer $k\ge 1$, there are finitely many $k$-vertex-critical $(P_5,P_2+2P_1)$-free graphs.
\end{theorem}

\begin{proof}
Let $G$ be a $k$-vertex-critical $(P_5,P_2+2P_1)$-free graph. If $G$ contains a $K_{k}$, then $G$ is isomorphic to $K_k$.
So we assume in the following that $G$ is $K_{k}$-free.
Let $K=\{v_1,\ldots,v_t\}$ be a maximal clique, where $1\le t< k$.
Since $K$ is maximal, every vertex in $V\setminus K$ is not adjacent to at least one vertex in $K$.
We partition $V\setminus K$ into the following subsets.

\begin{itemize}
\item $F_1$ is the set of nonneighbors of $v_1$.
\item  For $2\le i\le t$, $F_i$ is the set of nonneighbors of $v_i$ that are not in $F_1\cup \cdots \cup F_{i-1}$.
\end{itemize}
By the definition, $v_i$ is complete to $F_j$ if $i<j$.
Since $G$ is $2P_1+P_2$-free,  each $F_i$ is $P_1+P_2$-free, and so is a complete multipartite graph.
Since $G$ has no $K_k$, each $F_i$ has at most $k$ parts.

\begin{claim}\label{clm:2P2-free}
Let $S$ be a part of $F_i$ and $T$ be a part in $F_j$ with $i<j$. Then $G[S\cup T]$ is a $2P_2$-free graph.
\end{claim}

\begin{proof}
Suppose not. Let $s_1t_1$ and $s_2t_2$ be an induced $2P_2$, where $s_i\in S$ and $t_i\in T$ for $i=1,2$.
Then since $v_i$ is not adjacent to $s_1,s_2$ and is adjacent to $t_1,t_2$, it follows that $s_1,t_1,v_i,t_2,s_2$ induces a $P_5$,
a contradiction.
\end{proof}

\begin{claim}\label{clm:type}
Let $S$ be a part of $F_i$ and $T$ be a part in $F_j$ with $i<j$.  Every vertex in $T$ is adjacent to all but at most one vertex in $S$.
\end{claim}

\begin{proof}
Suppose that $t\in T$ is not adjacent to two vertices $s,s'$ in $S$. Then $\{v_i,s,s',t\}$ induces a $2P_1+P_2$, a contradiction.
\end{proof}

Next we show that each part in $F_i$ has bounded size.
\begin{claim}\label{clm:last}
$F_t$ is an independent set of bounded size.
\end{claim}

\begin{proof}
If $F_t$ has at least two parts, then any two vertices from two different parts and $K\setminus \{v_t\}$ form a clique of  size $|K|+1$,
contradicting the choice of $K$. So $F_i$ is an independent set. By \autoref{clm:type}, each vertex in $F_t$ is adjacent to all but at most
one vertex in any part of $F_i$ with $1\le i\le t-1$. For each part $S$ in $F_1\cup \cdots \cup F_{t-1}$, we introduce a binary variable
$X_S\in \{0,1\}$. If $X_S=0$, it indicates that a vertex in $F_t$ is complete to $S$ while $X_S=1$ indicates that a vertex in $F_t$
is adjacent to all vertices in $S$ except one vertex.
 A {\em type} is a binary vector $(X_S)_{S \text{ is a part of  } F_i \text{ with } i<t}$. Since the number of parts in each $F_i$ is at most $k$,
there are at most $2^{kt}\le 2^{k^2}$ types. If $|F_t|>2^{k^2}$, then there are two vertices $x,y\in F_t$ that have the same type. 
Let us fix a part $S\in F_1\cup \cdots \cup F_{t-1}$. 
If $X_S=0$, then both $x$ and $y$ are complete to $S$. If $X_S=1$, then each of $x$ and $y$ has a unique nonneighbor $x'$ and $y'$ 
in $S$. If $x'\neq y'$, then $\{x,x',y,y'\}$ induces a $2P_2$,  which contradicts \autoref{clm:2P2-free}.
So $x'=y'$ and thus $x$ and $y$ have the same neighbors in $S$.  Since $x$ and $y$ have the same type, $N(x)=N(y)$, contradicting 
\autoref{lem:dominated subsets}.
\end{proof}

\begin{claim}\label{clm:all}
For each $1\le i\le t$, $F_i$ has bounded size. 
\end{claim}

\begin{proof}
The statement is true for $i=t$ by \autoref{clm:last}. Now suppose that $i<t$ and $F_j$ has bounded size for each $i< j\le t$.
For each part $S$ in $F_1\cup \cdots \cup F_{i-1}$, we introduce a binary variable $X_S\in \{0,1\}$. 
Moreover, for each vertex $u$ in $\{v_j\}\cup F_j$ for $j>i$, we introduce a binary variable $X_{\{u\}}\in \{0,1\}$.
The meaning of $X_{\{u\}}$ is to indicate whether a vertex in $F_i$ is a neighbor or a nonneighbor of $u$.
A {\em type} is a binary vector 
\[(X_S)_{S \text{ is a part of  } F_{\ell} \text{ with } \ell<i \text{ or } S=\{u\} \text{ for some vertex } u\in \{v_j\}\cup F_j \text{ with }  j>i}.\]
By the inductive hypothesis, each $F_j$ with $j>i$ has bounded size. Therefore, there is a constant $M$ depending only on $k$
such that the number of types is at most $M$. If a part $T$ in $F_i$ has size larger than $M$, there are two vertices $x,y\in T$
having the same type. Using the exact same argument in \autoref{clm:last}, it follows that $N(x)=N(y)$.
This contradicts \autoref{lem:dominated subsets}.  Therefore, each part in $F_i$ has bounded size and so does $F_i$.
\end{proof}

By \autoref{clm:last} and \autoref{clm:all}, $G$ has bounded size.
\end{proof}

\subsection{Diamond-Free Graphs}

\begin{theorem}\label{thm:diamond}
For every fixed integer $k\ge 1$, there are finitely many $k$-vertex-critical ($P_5$, diamond)-free graphs.
\end{theorem}

\begin{proof}
Let $G$ be a $k$-vertex-critical ($P_5$, diamond)-free graph. We show that $|G|\le \max\{k,57\}$.
If $G$ contains a $K_k$, then $G$ is isomorphic to $K_k$ and thus $|G|\le \max\{k,57\}$.
So assume that $G$ is $K_k$-free. Since $G$ is imperfect, $G$ contains an induced $C_5$ by \autoref{thm:SPGT}.
Let $C=v_1,v_2,v_3,v_4,v_5$ be an induced $C_5$. For each $1\le i\le 5$, we define
\begin{equation*} \label{eq1}
\begin{split}
Z & = \{v\in V\setminus C: N_{C}(v)=\emptyset\}, \\
R_i  & = \{v\in V\setminus C: N_{C}(v)=\{v_{i-1},v_{i+1}\}\}, \\
Y_i  & = \{v\in V\setminus C: N_{C}(v)=\{v_{i-2}, v_i, v_{i+2}\}\}. \\
\end{split}
\end{equation*}
Let $R=\cup_{1\le i\le 5}R_i$ and $Y=\cup_{1\le i\le 5}Y_i$.

\begin{claim}\label{clm:partition}
$V(G)=V(C)\cup Z\cup R\cup Y$.
\end{claim}

\begin{proof}
Let $v\in V(G)\setminus V(C)$. If $v$ has three consecutive neighbors $v_i,v_{i+1},v_{i+2}$ on $C$, then $\{v,v_i,v_{i+1},v_{i+2}\}$
induces a diamond. So if $v$ has at least three neighbors on $C$, $v\in Y$. If $v$ has no neighbors on $C$, then $v\in Z$. 
Now assume that $1\le |N(v)\cap C|\le 2$. 
If $N(v)\subseteq \{v_{i-2},v_{i+2}\}$ for some $i$, say $v$ is adjacent to $v_{i+2}$, then $C\setminus \{v_{i-2}\}\cup \{v\}$
induces a $P_5$. So $v\in C$. This completes the proof.
\end{proof}

We first bound $Y$.

\begin{claim}\label{clm:stable}
Each $R_i$ and $Y_i$ is an independent set.
\end{claim}

\begin{proof}
Suppose that $R_i$ contains two adjacent vertices $x$ and $y$, then $\{x,y,v_{i-1},v_{i+1}\}$ induces a diamond.
The proof for $Y_i$ is the same.
\end{proof}

\begin{claim}\label{clm:Yi}
For each $1\le i\le 5$, $|Y_i|\le 1$.
\end{claim}

\begin{proof}
If $Y_i$ contains two nonadjacent vertices $x,y$, then $\{x,y,v_{i-2},v_{i+2}\}$ induces a diamond. So $Y_i$ is a clique.
By \autoref{clm:stable}, $|Y_i|\le 1$.
\end{proof}

Next we bound $Z$.

\begin{claim}\label{clm:ZR}
$Z$ is anticomplete to $R$.
\end{claim}

\begin{proof}
Let $z\in Z$. If $z$ has a neighbor $r\in R_1$, then $z,r,v_2,v_3,v_4$ induces a $P_5$.
\end{proof}

\begin{claim}\label{clm:ZY}
Each vertex in $Y$ is either complete or anticomplete to a component of $Z$.
\end{claim}

\begin{proof}
Let $y\in Y_1$. If $y$ is mixed on an edge $wz$ in $Z$ with $yw\notin E$ and $yz\in E$, then $w,z,y,v_4,v_5$ induces a $P_5$.
\end{proof}

\begin{claim}
$|Z|\le 32$.
\end{claim}

\begin{proof}
We first show that $Z$ is an independent set. Let $Q$ be any component of $Z$. Then $N(Q)\subseteq Y$ by \autoref{clm:ZR}.
By \autoref{lem:clique cutsets}, $N(Q)$ contains two nonadjacent vertices $y,y'\in Y$. By \autoref{clm:ZY}, $\{y,y'\}$ is complete to $Q$.
Since $G$ is diamond-free, $Q$ is a singleton. This proves that $Z$ is an independent set. If $|Z|>32$, then there are two vertices in $Z$
that have the same neighborhood by \autoref{clm:Yi}. This contradicts \autoref{lem:dominated subsets}.
\end{proof}

Finally, we bound $R$.

\begin{claim}\label{clm:RiRi+1}
$R_i$ and $R_{i+1}$ are complete to each other.
\end{claim}

\begin{proof}
Let $r_3\in R_3$ and $r_4\in R_4$. If $r_3r_4\notin E$, then $r_4,v_5,v_1,v_2,r_3$ induces a $P_5$.
\end{proof}

\begin{claim}\label{clm:RiRi+2}
$G[R_i\cup R_{i+2}]$ contains at most one edge.
\end{claim}

\begin{proof}
By symmetry, we prove for $i=1$. Let $r\in R_1$. If $r$ has two neighbors in $R_3$, then these two vertices together with $v_2,r$
induce a diamond by \autoref{clm:stable}. So every vertex in $R_1$ has at most one neighbor in $R_3$. Similarly, every vertex in $R_3$
has at most one neighbor in $R_1$. If $G[R_1\cup R_3]$ contains two edges $xy$ and $x'y'$ with $x,x'\in R_1$ and $y,y'\in R_3$,
then $y,x,v_5,x',y'$ induce a $P_5$.
\end{proof}

\begin{claim}\label{clm:RiYj}
$R_i$ is complete to $Y_i$ and is anticomplete to $Y_j$ for $j\neq i$.
\end{claim}

\begin{proof}
Let $r\in R_1$. If $r$ is not adjacent to $y\in Y_1$, $y,v_3,v_2,r,v_5$ induces a $P_5$.
If $r$ is adjacent to $y\in Y_2$, then $\{r,y,v_2,v_5\}$ induces a diamond. 
If $r$ is adjacent to $y\in Y_3$, then $\{r,y,v_1,v_5\}$ induces a diamond. 
This completes the proof.
\end{proof}

\begin{claim}\label{clm:R}
For each $1\le i\le 5$, $|R_i|\le 3$.
\end{claim}

\begin{proof}
Suppose that $|R_1|\ge 4$. Then by Claims \ref{clm:RiRi+1}-\ref{clm:RiYj},  there are two vertices
in $R_1$ that have the same neighborhood in $G$. This contradicts \autoref{lem:dominated subsets}.
\end{proof}

By \autoref{clm:partition}, it follows that $|G|=|V(C)|+|Y|+|R|+|Z|\le 5+5+15+32=57$.
This proves the theorem.
\end{proof}

\section{A Complete Classification}\label{sec:classification}

In this section, we prove our main result \autoref{thm:main}.

\begin{proof}[Proof of \autoref{thm:main}]
An infinite family of 5-vertex-critical $2P_2$-free graphs is constructed in~\cite{HMRSV15}. It can be easily checked that
these graphs are $P_1+K_3$-free. Since $2P_2$ and $P_1+K_3$ do not contain any universal vertices, for every fixed $k\ge 6$
one can obtain an infinite family of $k$-vertex-critical $2P_2$-free graphs and $(P_5,P_1+K_3)$-free graphs by 
adding $k-5$ universal vertices to the 5-vertex-critical family in~\cite{HMRSV15}.

Now assume that $H$ is not $2P_2$ or $P_1+K_3$. 
Let $G$ be a $k$-vertex-critical $(P_5,H)$-free graph. We may assume that $G$ is $K_k$-free
for otherwise $G$ is $K_k$.  If $H=4P_1$, then Ramsey's theorem~\cite{Ram30} shows that $|G|\le R(4,k)-1$.
If $H=K_4$, then there are no $k$-vertex-critical $(P_5,K_4)$-free graphs for any $k\ge 6$~\cite{ELMM13}. Moreover,
there are only two 5-vertex-critical $(P_5,K_4)$-free graphs by \autoref{thm:K4}.
If $H$ is a diamond or $P_2+2P_1$, then the finiteness follows from \autoref{thm:diamond} and \autoref{thm:co-diamond}, respectively.
If $H=C_4$, then the finiteness follows from~\cite{HH17}.
If $H=P_4$, then $G$ is perfect and so $(k-1)$-colorable, a contradiction. 
If $H$ is a claw, then the finiteness follows from~\cite{LR03}.
If $H$ is $P_1+P_3$, then the finiteness follows from \autoref{thm:co-paw}.
If $H$ is a paw, then $G$ is either triangle-free or a complete multipartite graph by a result of Olariu~\cite{Ol88}.
In either case, $G$ is $(k-1)$-colorable, a contradiction.
\end{proof}

In view of \autoref{thm:main}, it is natural to ask the following question, which we leave as a possible future direction.

\noindent {\bf Problem.} Which five-vertex graphs $H$ could lead to finitely many $k$-vertex-critical $(P_5,H)$-free
graphs?

As mentioned in the introduction, it was shown in \cite{DHHMMP17} that $H=\overline{P_5}$ is one such graph. 

\section{Appendix}

The source code of \autoref{alg:main} and \autoref{alg:extend} which 
we used in the proofs of \autoref{lem:W5F} and \autoref{lem:7-antihole-free} can be downloaded from~\cite{criticalpfree-site}. 
We refer to~\cite{GO18} for more details on how we verified the correctness of our implementation. 
We executed the program on an Intel i9-9900 CPU at 3.10GHz and in each case the program terminated in a few seconds.

Below we give the adjacency list of the two 5-vertex-critical $(P_5,K_4)$-free graphs $G_1$ and $G_2$ from \autoref{thm:K4oddhole}. They can also be obtained from the database of interesting graphs at the \textit{House of Graphs}~\cite{hog} by searching for the keywords ``5-critical P5K4-free''\footnote{These graphs can also be accessed directly at \url{https://hog.grinvin.org/ViewGraphInfo.action?id=34489} and \url{https://hog.grinvin.org/ViewGraphInfo.action?id=34491}}. 

\begin{itemize}
\item Graph $G_1$: \{0: 1 2 10 12; 1: 0 8 10 12; 2: 0 9 10 12; 3: 4 5 11 12; 4: 3 6 11 12; 5: 3 7 11 12; 6: 4 7 11 12; 7: 5 6 11 12; 8: 1 9 10 12; 9: 2 8 10 12; 10: 0 1 2 8 9 11; 11: 3 4 5 6 7 10; 12: 0 1 2 3 4 5 6 7 8 9\}
\item Graph $G_2$: \{0: 1 2 12 13; 1: 0 3 12 13; 2: 0 4 12 13; 3: 1 4 12 13; 4: 2 3 12 13; 5: 6 7 9 11 12; 6: 5 8 10 11 13; 7: 5 8 9 11 13; 8: 6 7 10 11 12; 9: 5 7 10 12 13; 10: 6 8 9 12 13; 11: 5 6 7 8 12 13; 12: 0 1 2 3 4 5 8 9 10 11; 13: 0 1 2 3 4 6 7 9 10 11\}
\end{itemize}



\end{document}